\documentclass[a4paper,reqno,12pt]{amsart}
\usepackage{verbatim}
\usepackage{amssymb,amsmath,amsthm}
\usepackage{amsfonts}
\usepackage{array}

\newtheorem{theorem}{Theorem}

\theoremstyle{remark}

\begin{document}

\setcounter{page}{1}

\title[]{A bivariate generating function for zeta values \\
and related supercongruences
}

\author{Roberto~Tauraso}

\address{Dipartimento di Matematica, 
Universit\`a di Roma ``Tor Vergata'', 
via della Ricerca Scientifica, 
00133 Roma, Italy}
\email{tauraso@mat.uniroma2.it}



\date{\today}

\begin{abstract} 
By using the Wilf-Zeilberger method, we prove a novel finite combinatorial identity related to a bivariate generating function for $\zeta(2+r+2s)$ (an extension of a Bailey-Borwein-Bradley Ap\'ery-like formula for even zeta values). Such identity is then applied to show several supercongruences.
\end{abstract}

\maketitle

\section{Introduction} 
The  bivariate formula 
\begin{equation}\label{GF1}
\sum_{k=1}^{\infty}\frac{k}{k^4-a^2k^2-b^4}
=\frac{1}{2}\sum_{k=1}^{\infty}\frac{(-1)^{k-1}(5k^2-a^2)}{k\binom{2k}{k}}
\cdot\frac{\prod_{j=1}^{k-1}((j^2-a^2)^2+4b^4)}{\prod_{j=1}^{k} (j^4-a^2j^2-b^4)}
\end{equation}
has been first conjectured by H. Cohen and then proved independently by T. Rivoal \cite[Theorem 1.1]{Ri04} and  D. M. Bradley \cite[Theorem 1]{Br08} by reducing it to the finite combinatorial identity 
\begin{equation}\label{ID1}
\sum_{k=1}^n\binom{2k}{k}
\frac{(5k^2-a^2)\prod_{j=1}^{k-1}((n^2-j^2)(n^2+j^2-a^2))}{
\prod_{j=1}^{k} (n^2+(n-j)^2-a^2)(n^2+(n+j)^2-a^2)}=\frac{2}{n^2-a^2}
\end{equation}
and by Kh. and T. Hessami Pilehrood \cite[Theorem 1]{PP08} by applying the Wilf-Zeilberger theory. Since the left-hand side of \eqref{GF1} can be written as the generating function of $\zeta(3+2r+4s)$,
$$\sum_{r=0}^{\infty}\sum_{s=0}^{\infty}\binom{r+s}{r}\zeta(3+2r+4s)a^{2r}b^{4s},$$
it follows that, by extracting the coefficients  for $(r,s)=(0,0)$
and $(r,s)=(1,0)$, we obtain the  A\'pery-like identities
\begin{equation}\label{S12}
\zeta(3)=\frac{5}{2}\sum_{k=1}^{\infty}\frac{(-1)^{k-1}}{k^3\binom{2k}{k}}
\quad\mbox{and}\quad
\zeta(5)=\frac{1}{2}\sum_{k=1}^{\infty}\frac{(-1)^{k-1}}{\binom{2k}{k}}
\left(\frac{4}{k^5}-\frac{5H_{k-1}(2)}{k^3}\right)
\end{equation}
where $H_n(s)=\sum_{j=1}^n\frac{1}{j^s}$ is the {\sl harmonic sum} of weight $s$. For more details about Ap\'ery-like series see also \cite{AG99,BB97,PP08,Ch17}.

Here we consider a similar bivariate formula 
\begin{equation}\label{GF2}
\sum_{k=1}^{\infty}\frac{1}{k^2-ak-b^2}
=\sum_{k=1}^{\infty}\frac{(3k-a)}{k\binom{2k}{k}}
\cdot\frac{\prod_{j=1}^{k-1}(j^2-a^2-4b^2)}{\prod_{j=1}^{k} (j^2-aj-b^2)}
\end{equation}
where the left-hand side is the generating function of $\zeta(2+r+2s)$,
$$\sum_{r=0}^{\infty}\sum_{s=0}^{\infty}
\binom{r+s}{r}\zeta(2+r+2s)a^rb^{2s}.$$
For $a=0$, \eqref{GF2} yields a formula due to D. H. Bailey, J. M. Borwein, and D. M. Bradley, 
$$\sum_{s=0}^{\infty}
\zeta(2+2s)b^{2s}=\sum_{k=1}^{\infty}\frac{1}{k^2-b^2}
=3\sum_{k=1}^{\infty}\frac{1}{\binom{2k}{k}}
\cdot\frac{\prod_{j=1}^{k-1}(j^2-4b^2)}{\prod_{j=1}^{k} (j^2-b^2)}$$
which appeared in \cite[Theorem 1.1]{BBB06}. 
Moreover, for $(r,s)=(1,0)$ and $(r,s)=(0,1)$, we get the A\'pery-like identities
\begin{align}\label{S34}
\zeta(3)=\sum_{k=1}^{\infty}\frac{1}{\binom{2k}{k}}
\left(\frac{2}{k^3}+\frac{3H_{k-1}(1)}{k^2}\right)\quad\mbox{and}\quad
\zeta(4)=3\sum_{k=1}^{\infty}\frac{1}{\binom{2k}{k}}
\left(\frac{1}{k^4}-\frac{3H_{k-1}(2)}{k^2}\right).
\end{align}
Replacing $a$ by $2a$ and then letting $x^2 = a^2 + b^2$
in \eqref{GF2} we find the equivalent identity
$$\sum_{k=1}^{\infty}\frac{1}{(k-a)^2-x^2}
=\sum_{k=1}^{\infty}\frac{(3k-2a)}{k\binom{2k}{k}}
\cdot\frac{\prod_{j=1}^{k-1}(j^2-4x^2)}{\prod_{j=1}^{k} ((j-a)-x^2)}$$
which has been proved by Kh. and T. Hessami Pilehrood \cite[(24)]{PP12a}.

Again, in the same spirit of what has been done for \eqref{GF1}, our proof of the identity \eqref{GF2} is reduced to show the following novel finite identity
\begin{equation}\label{ID2}
\sum_{k=1}^n
\binom{2k}{k}\frac{3k-2n+a}{k^2-a^2}\cdot\prod_{j=1}^{k-1}
\frac{(j-n)(j-n+a)}{j^2-a^2}=\frac{2}{n-a}.
\end{equation}
In \cite[Theorem 4.2]{Ta10} the author established that for any prime $p>5$,
\begin{align}\label{CO1}
&\sum_{k=1}^{p-1}\frac{1}{k}\binom{2k}{k}
\equiv -\frac{8H_{p-1}(1)}{3}\pmod{p^4},\\
\label{CO2}
&\sum_{k=1}^{p-1}\frac{(-1)^{k}}{k^2}\binom{2k}{k}
\equiv \frac{4}{5}\left(\frac{H_{p-1}(1)}{p}+2pH_{p-1}(3)\right)\pmod{p^4}.
\end{align}
Thanks to the finite identities \eqref{ID1} and \eqref{ID2}, we managed to improve congruence \eqref{CO1} and to show several other congruences. The main results are as follows: for any prime $p>5$,
\begin{align}
\label{CO4b}
&\sum_{k=1}^{p-1}\frac{1}{k^3}\binom{2k}{k}\equiv 
-\frac{2H_{p-1}(1)}{p^2}\pmod{p^2},\\
\label{CO4c}
&\sum_{k=1}^{p-1}\binom{2k}{k}\frac{H_k(2)}{k}
\equiv \frac{2H_{p-1}(1)}{3p^2}\pmod{p^2}.
\end{align}
These congruences are known modulo $p$ (see \cite[Theorem 2]{PP12}) and they confirm modulo $p^2$ the following conjecture by Z.-W. Sun 
: for each prime $p>7$,
\begin{align*}
&\sum_{k=1}^{p-1}\frac{1}{k^3}\binom{2k}{k}\equiv 
-\frac{2H_{p-1}(1)}{p^2}-\frac{13H_{p-1}(3)}{27}\pmod{p^4},\\
&\sum_{k=1}^{p-1}\binom{2k}{k}\frac{H_k(2)}{k}
\equiv \frac{2H_{p-1}(1)}{3p^2}-\frac{38H_{p-1}(3)}{81}\pmod{p^3}.
\end{align*}
the first one appeared in \cite[Conjecture 1.1]{SunZW11} and the second one in \cite[Conjecture 5.1]{SunZW15}.

\section{Preliminaries concerning multiple harmonic sums}

We define the {\sl multiple harmonic sum} as
$$H_n(s_1,\dots,s_r)=\sum_{1\leq k_1<k_2<\cdots<k_r\leq n}\frac{1}{k_1^{s_1} k_2^{s_2}\cdots k_r^{s_r}}$$
where $n\geq r>0$ and each $s_i$ is a positive integer. The sum $s_1+s_2+\dots+s_r$ is the  weight of the multiple sum. Furthermore, by $\{s_1, s_2, \dots, s_j\}^m$ we denote the
sequence of length $mj$ with $m$ repetitions of $(s_1, s_2,\dots, s_j)$.

\noindent By \cite[Theorem 5.1]{SunZH00}), for any prime $p>s+2$ we have
\begin{align*}
H_{p-1}(s)\equiv
\begin{cases}
\displaystyle -\frac{s(s+1)}{2(s+2)}\,p^2\,B_{p-s-2}  \pmod{p^3} &\mbox{if $s$ is odd,}\vspace{1mm}\\
\displaystyle \frac{s}{s+1}\,p\,B_{p-s-1}  \pmod{p^2} &\mbox{if $s$ is even.}
\end{cases}
\end{align*}
where $B_n$ be the $n$-th Bernoulli number.

\noindent Let $p>5$ be a prime,  then by \cite[Theorem 2.1]{Ta10},
\begin{equation}\label{ta}
 H_{p-1}(2)\equiv  - \frac{2H_{p-1}(1)}{p}-\frac{pH_{p-1}(3)}{3}\pmod{p^4}.
 \end{equation}
Moreover, by \cite[Lemma 3]{PP12},
$$H_{p-1}(1,2)\equiv  -\frac{3H_{p-1}(1)}{p^2}-\frac{5H_{p-1}(3)}{12}\pmod{p^3}$$
and by \cite[Proposition 3.7]{Za08} and  \cite[Theorem 4.5]{PPT14}
$$H_{p-1}(1,1,2)\equiv -\frac{11H_{p-1}(3)}{12p}\pmod{p^2}\;,\;
H_{p-1}(1,1,1,2)\equiv -\frac{5H_{p-1}(3)}{6p^2}\pmod{p}.$$
Finally, by \cite[Theorem 3.2]{Za08},
$$H_{p-1}(2,2)\equiv \frac{H_{p-1}(3)}{3p}\;,\;
H_{p-1}(1,3)\equiv \frac{3H_{p-1}(3)}{4p} \pmod{p^2}$$
and by \cite[Theorem 3.5]{Za08}, 
$$H_{p-1}(2,1,2)\equiv 0\;,\;
H_{p-1}(1,2,2)\equiv \frac{5H_{p-1}(3)}{4p^2}\;,\;
H_{p-1}(1,1,3)\equiv -\frac{5H_{p-1}(3)}{12p^2} \pmod{p}.$$

\section{Proofs of the generating function \eqref{GF2} and the related combinatorial identity \eqref{ID2}}

By partial fraction decomposition with respect to $b^2$, we get
$$\frac{\prod_{j=1}^{k-1}(j^2-a^2-4b^2)}{\prod_{j=1}^{k} (j^2-aj-b^2)}=
\sum_{n=1}^k\frac{C_{n,k}(a)}{n^2-an-b^2}$$
where
$$C_{n,k}(a)=\frac{\prod_{j=1}^{k-1} (j^2-(a-2n)^2)}{\prod_{j=1,j\not=n}^{k} (j-n)(j+n-a)}.$$
Hence, by inverting the summations order, the identity \eqref{GF2} can be written as
$$\sum_{n=1}^{\infty}\frac{1}{n^2-an-b^2}
=\sum_{k=1}^{\infty}\frac{(3k-a)}{k\binom{2k}{k}}
\sum_{n=1}^k\frac{C_{n,k}(a)}{n^2-an-b^2}
=\sum_{n=1}^{\infty}\frac{1}{n^2-an-b^2}
\sum_{k=n}^{\infty}\frac{(3k-a)C_{n,k}(a)}{k\binom{2k}{k}}.$$
It follows that \eqref{GF2} holds as soon as
\begin{equation}\label{x1}
1=\sum_{k=n}^{\infty}\frac{(3k-a)C_{n,k}(a)}{k\binom{2k}{k}}=
\sum_{k=n}^{\infty}\frac{(3k-a)}{k\binom{2k}{k}}\cdot\frac{\prod_{j=1}^{k-1} (j^2-(a-2n)^2)}{\prod_{j=1,j\not=n}^{k} (j-n)(j+n-a)}.
\end{equation}
Taking the same approach given in \cite{Ri04} for the proof of \eqref{GF1}, the above formula is equivalent to this finite combinatorial identity 
\begin{equation}\label{x2}
\sum_{k=1}^n
\binom{2k}{k}(3k-a)\,
\frac{\prod_{j=1}^{k-1}(j-n)(j+n-a)}{\prod_{j=1}^{k}(j^2-(a-2n)^2)}=\frac{2}{a-n}.
\end{equation}
Both identities \eqref{x1} and \eqref{x2} are consequences of the next theorem after setting $z=2n-a$.
\begin{theorem} For any positive integer $n$,
\begin{equation}\label{y2}
\sum_{k=1}^n
\binom{2k}{k}(3k-2n+z)\,
\frac{\prod_{j=1}^{k-1}(j-n)(j-n+z)}{\prod_{j=1}^{k} (j^2-z^2)}=\frac{2}{n-z},
\end{equation}
and
\begin{equation}\label{y1}
\sum_{k=n}^{\infty}\frac{(3k-2n+z)}{k\binom{2k}{k}}\cdot\frac{\prod_{j=1}^{k-1} (j^2-z^2)}{\prod_{j=1,j\not=n}^{k} (j-n)(j-n+z)}=1.
\end{equation}
\end{theorem}
\begin{proof}
Let
$$F(n,k)=\binom{2k}{k}(3k-2n+z)
\frac{\prod_{j=0}^{k-1}(j-n)(j-n+z)}{\prod_{j=1}^{k}(j^2-z^2)},
$$
and
$$G(n,k)=\frac{k(k^2-z^2)F(n,k)}{(2n-3k-z)(n+1-k)(n+1-k-z)}.$$
Then $(F,G)$ is a Wilf-Zeilberger pair, or WZ pair, which means that they satisfy the relation
$$F(n+1,k)-F(n,k)=G(n,k+1)-G(n,k).$$
In order to prove \eqref{y2}, it suffices to prove that $S_n:=\sum_{k=1}^{n}F(n,k)=2n$.  Now
$S_1=F(1,1)=2$.
Morever, 
$$S_{n+1}-S_n=\sum_{k=1}^{n+1}F(n+1,k)-\sum_{k=1}^{n+1}F(n,k)=G(n,n+2)-G(n,1)=2$$
because $F(n,n+1)=G(n,n+2)=0$ and $G(n,1)=-2$.

In a similar way, we show \eqref{y1} by considering the WZ pair  given by
$$F(n,k)=\frac{(3k-2n+z)}{k\binom{2k}{k}}\cdot\frac{\prod_{j=1}^{k-1} (j^2-z^2)}{\prod_{j=1,j\not=n}^{k} (j-n)(j-n+z)},
$$
and
$$G(n,k)=\frac{2(2k-1)(k-n)F(n,k)}{n(2n-3k-z)(n-z)}.$$
\end{proof}

\section{More binomial identities}

Here we collect a few identities, apparently new, involving the binomial coefficients $\binom{2k}{k}$ and $\binom{n+k}{k}$ which will play a crucial role in the next sections. 

\begin{theorem} For any positive integer $n$,
\begin{align}\label{Id0}
&\frac{3}{2}\sum_{k=1}^{n} \frac{1}{k}\binom{2k}{k}=
\sum_{k=1}^{n}\frac{1}{k}\binom{n+k}{k}+H_{n}(1)\\
\label{Id1}
&\sum_{k=1}^n\binom{2k}{k}\left(\frac{3H_{k}(1)}{2k}-\frac{1}{k^2}\right)
=\sum_{k=1}^{n}\binom{n+k}{k}\frac{H_k(1)}{k}-H_n(2)\\
\label{Id2}
&\sum_{k=1}^n\binom{2k}{k}\left(\frac{3H_{k}(2)}{k}-\frac{1}{2k^3}\right)
=\sum_{k=1}^{n}\binom{n+k}{k}\frac{H_k(2)+H_n(2)}{k}
+H_n(2)H_n(1)-H_n(1,2)
\end{align}
\end{theorem}
\begin{proof} 
Let us consider the WZ pair
$$F(n,k)=\frac{1}{k}\binom{n+k}{k}\quad\mbox{and}\quad
G(n,k)=\frac{k}{(n+1)^2}\binom{n+k}{k}
$$ 
then 
\begin{align*}
S_{n+1}-S_n&=F(n+1,n+1)+\sum_{k=1}^{n}(G(n,k+1)-G(n,k))\\
&=F(n+1,n+1)+G(n,n+1)-G(n,1)\\
&=\frac{3/2}{n+1}\binom{2(n+1)}{n+1}-\frac{1}{n+1}
\end{align*}
where $S_n:=\sum_{k=1}^{n}F(n,k)$. Thus
\begin{align*}
S_n=\frac{3}{2}\sum_{k=1}^n\frac{1}{k}\binom{2k}{k}-H_n(1)
\end{align*}
and we may conclude that \eqref{Id0} holds.

\noindent Now let  $S_n^{(1)}:=\sum_{k=1}^{n}F(n,k)H_k(1)$ then
\begin{align*}
S_{n+1}^{(1)}-S_n^{(1)}&=F(n+1,n+1)H_{n+1}(1)\\
&\qquad+\sum_{k=1}^{n}\left(G(n,k+1)H_{k}(1)-G(n,k)\left(H_{k-1}(1)+\frac{1}{k}\right)\right)\\
&=F(n+1,n+1)H_{n+1}(1)+G(n,n+1)H_{n}(1)-\sum_{k=1}^{n}\frac{G(n,k)}{k}\\
&=\binom{2(n+1)}{n+1}\left(\frac{3H_{n+1}(1)}{2(n+1)}-\frac{1}{(n+1)^2}\right)+\frac{1}{(n+1)^2}
\end{align*}
where we used $\sum_{k=1}^n\binom{n+k}{k}=\frac{1}{2}\binom{2(n+1)}{n+1}-1$. Hence we find that
$$S_n^{(1)}=\sum_{k=1}^n\binom{2k}{k}\left(\frac{3H_{k}(1)}{2k}-\frac{1}{k^2}\right)+H_n(2)$$
which implies \eqref{Id1}.

\noindent Let  $S_n^{(2)}:=\sum_{k=1}^{n}F(n,k)H_k(2)$ then
\begin{align*}
S_{n+1}^{(2)}-S_n^{(2)}&=F(n+1,n+1)H_{n+1}(2)\\
&\qquad+\sum_{k=1}^{n}\left(G(n,k+1)H_{k}(2)-G(n,k)\left(H_{k-1}(2)+\frac{1}{k^2}\right)\right)\\
&=F(n+1,n+1)H_{n+1}(2)+G(n,n+1)H_{n}(2)-\sum_{k=1}^{n}\frac{G(n,k)}{k^2}\\
&=\binom{2(n+1)}{n+1}\left(\frac{3H_{n+1}(2)}{2(n+1)}-\frac{1}{2(n+1)^3}\right)-\frac{S_n}{(n+1)^2}
\end{align*}
where we applied
$$
\sum_{k=1}^{n}\frac{G(n,k)}{k^2}=\frac{1}{(n+1)^2}\sum_{k=1}^{n}F(n,k)=\frac{S_n}{(n+1)^2}.$$ Therefore
\begin{align*}
S_n^{(2)}&=\sum_{k=1}^n\binom{2k}{k}\left(\frac{3H_{k}(2)}{2k}-\frac{1}{2k^3}\right)-\sum_{k=1}^n\frac{S_{k-1}}{k^2}\\
&=\sum_{k=1}^n\binom{2k}{k}\left(\frac{3H_{k}(2)}{2k}-\frac{1}{2k^3}\right)-\frac{3}{2}\sum_{k=1}^n\frac{1}{k^2}\sum_{j=1}^{k-1}\frac{1}{j}\binom{2j}{j}+H_n(1,2)\\
&=\sum_{k=1}^n\binom{2k}{k}\left(\frac{3H_{k}(2)}{2k}-\frac{1}{2k^3}\right)-\frac{3}{2}\sum_{j=1}^{n }\frac{1}{j}\binom{2j}{j}(H_n(2)-H_j(2))+H_n(1,2)\\
&=\sum_{k=1}^n\binom{2k}{k}\left(\frac{3H_{k}(2)}{k}-\frac{1}{2k^3}\right)-\frac{3H_n(2)}{2}\sum_{k=1}^{n}\frac{1}{k}\binom{2k}{k}+H_n(1,2)\\
\end{align*}
and the proof of  \eqref{Id2} is complete.
\end{proof}

\section{Proofs of the main supercongruences} 

\begin{theorem} For any prime $p>3$,
\begin{equation}\label{CO5}
\sum_{k=1}^{p-1} \frac{1}{k}\binom{2k}{k}\equiv
-\frac{8H_{p-1}(1)}{3}-\frac{5p^2H_{p-1}(3)}{3}\pmod{p^5}\\
\end{equation}
Moreover, for any prime $p>5$,
\begin{align*}
&\sum_{k=1}^{p-1}\frac{1}{k^3}\binom{2k}{k}\equiv 
-\frac{2H_{p-1}(1)}{p^2}\pmod{p^2},\\
&\sum_{k=1}^{p-1}\binom{2k}{k}\frac{H_k(2)}{k}
\equiv \frac{2H_{p-1}(1)}{3p^2}\pmod{p^2}.
\end{align*}
\end{theorem}
\begin{proof} We first note that
\begin{equation}\label{nkk}
\binom{p-1+k}{k}=\frac{p}{k}\binom{p+k-1}{k-1}=
\frac{p}{k}\prod_{j=1}^{k-1}\left(1+\frac{p}{j}\right) 
=\frac{1}{k}\sum_{j=0}^{k-1}p^{j+1}H_{k-1}(\{1\}^{j}).
\end{equation}
Therefore, by \eqref{Id0} with $n=p-1$ we obtain the desired congruence \eqref{CO5},
\begin{align*}
\sum_{k=1}^{p-1} \frac{1}{k}\binom{2k}{k}&=\frac{2}{3}\left(H_{p-1}(1)+\sum_{j=0}^{p-2}p^{j+1}H_{p-1}(\{1\}^{j},2) \right)\\
&\equiv \frac{2}{3}\left(H_{p-1}(1)+pH_{p-1}(2)+p^2 H_{p-1}(1,2)\right.\\
&\qquad\qquad \left. +p^3 H_{p-1}(1,1,2)+p^4 H_{p-1}(1,1,1,2)\right)\\
&\equiv-\frac{8H_{p-1}(1)}{3}-\frac{5p^2H_{p-1}(3)}{3}\pmod{p^5}.
\end{align*}
By letting $z=2n$ in \eqref{y2} we have
$$\sum_{k=1}^n
\binom{2k}{k}\frac{k}{k^2-4n^2}\,
\prod_{j=1}^{k-1}\frac{j^2-n^2}{j^2-4n^2}=-\frac{2}{3n}.$$
Let $n=p>5$ be a prime and move the $p$-th term of the sum to the right-hand side,
$$\sum_{k=1}^{p-1}
\frac{1}{k}\binom{2k}{k}\frac{1}{1-\frac{4p^2}{k^2}}\,
\prod_{j=1}^{k-1}\frac{1-\frac{p^2}{j^2}}{1-\frac{4p^2}{j^2}}=\frac{2}{3p}\left(
\frac{1}{2}\binom{2p}{p}\prod_{j=1}^{p-1}\frac{1-\frac{p^2}{j^2}}{1-\frac{4p^2}{j^2}}-1\right).$$
The left-hand side modulo $p^4$ is congruent to
$$\sum_{k=1}^{p-1}
\frac{1}{k}\binom{2k}{k}\left(1+\frac{4p^2}{k^2}\right)\,
\prod_{j=1}^{k-1}\left(1+\frac{3p^2}{j^2}\right)\equiv
\sum_{k=1}^{p-1}
\frac{1}{k}\binom{2k}{k}+p^2\sum_{k=1}^{p-1}
\binom{2k}{k}\left(\frac{1}{k^3}+\frac{3H_k(2)}{k}\right).
$$
On the other hand, by \cite[Theorem 2.4]{Ta10}, 
\begin{equation}\label{ppp}
\frac{1}{2}\binom{2p}{p}\equiv 1+2pH_{p-1}(1)+\frac{2p^3H_{p-1}(3)}{3}
\equiv 1-p^2H_{p-1}(2)-\frac{p^4H_{p-1}(4)}{2}
\pmod{p^6},
\end{equation}
the right-hand side is
\begin{align*}
\frac{1}{2}\binom{2p}{p}\prod_{j=1}^{p-1}\frac{1-\frac{p^2}{j^2}}{1-\frac{4p^2}{j^2}}&\equiv \frac{1}{2}\binom{2p}{p}\prod_{j=1}^{p-1}\left(1+\frac{3p^2}{j^2}+\frac{12p^4}{j^4}\right)\\
&\equiv \left(1-p^2H_{p-1}(2)-\frac{p^4H_{p-1}(4)}{2}\right)
\\&\qquad \cdot\left(1+3p^2H_{p-1}(2)+12p^4H_{p-1}(4)+9p^4H_{p-1}(2,2)\right)\\
&\equiv1+2p^2H_{p-1}(2)+p^4\left(\frac{17H_{p-1}(4)}{2}+3H_{p-1}(2,2)\right)\\
&\equiv1+2p^2H_{p-1}(2)
\pmod{p^5}.\end{align*}
where $2H_{p-1}(2,2)=(H_{p-1}(2))^2-H_{p-1}(4)\equiv 0 \pmod{p}$.
Finally, by \eqref{CO5},
\begin{equation}\label{D1}
\sum_{k=1}^{p-1}
\binom{2k}{k}\left(\frac{1}{k^3}+\frac{3H_k(2)}{k}\right)\equiv
\frac{8H_{p-1}(1)}{3p^2}+\frac{5H_{p-1}(3)}{3}+\frac{4pH_{p-1}(2)}{3}
\equiv 0 \pmod{p^2}.
\end{equation}
where we used \eqref{ta}.

\noindent By \eqref{Id2}, with $n=p-1$, we have that
\begin{align}\label{D2}
\sum_{k=1}^{p-1}\binom{2k}{k}\left(\frac{3H_{k}(2)}{k}-\frac{1}{2k^3}\right)
&=p\sum_{k=1}^{p-1}\prod_{j=1}^{k-1}\left(1+\frac{p}{j}\right)
\frac{H_k(2)+H_{p-1}(2)}{k^2}\nonumber\\
&\qquad 
+H_{p-1}(2)H_{p-1}(1)-H_{p-1}(1,2)\nonumber\\
&\equiv p\sum_{k=1}^{p-1}\frac{H_k(2)}{k^2}-H_{p-1}(1,2)
\nonumber\\
&=pH_{p-1}(2,2)+pH_{p-1}(4) -H_{p-1}(1,2)\nonumber\\
&\equiv -H_{p-1}(1,2)=\frac{3H_{p-1}(1)}{p^2}\pmod{p^2}.
\end{align}
The proof is completety as soon as we combine properly congruences \eqref{D1} and \eqref{D2}.
\end{proof}

\section{Finale - Two A\'pery-like congruences} 

The following congruences are related to the second series in \eqref{S12} and to the first series in \eqref{S34}.

\begin{theorem} For any prime $p>3$,
\begin{align}\label{CO5b}
&\sum_{k=1}^{p-1}\binom{2k}{k}\left(\frac{2}{k^2}-\frac{3H_k(1)}{k}\right)
\equiv \frac{2H_{p-1}(1)}{p}+3pH_{p-1}(3)\pmod{p^4},\\
\label{CO5c}
&\sum_{k=1}^{p-1}(-1)^{k}\binom{2k}{k}\left(\frac{4}{k^4}+\frac{5H_k(2)}{k^2}\right)
\equiv -H_{p-1}(4) \pmod{p^2}.
\end{align}
\end{theorem}
\begin{proof} By \eqref{Id1} with $n=p-1$ and by $\eqref{nkk}$
\begin{align*}
\sum_{k=1}^{p-1}\binom{2k}{k}\left(\frac{3H_{k}(1)}{2k}-\frac{1}{k^2}\right)
&=\sum_{k=1}^{p-1}\frac{H_k(1)}{k^2}\sum_{j=0}^{k-1}p^{j+1}H_{k-1}(\{1\}^{j})-H_{p-1}(2)\\
&\equiv p\sum_{k=1}^{p-1}\frac{H_{k-1}(1)+\frac{1}{k}}{k^2}\left(1+pH_{k-1}(1)+p^2H_{k-1}(1,1)\right)-H_{p-1}(2)\\
&\equiv -H_{p-1}(2)+pH_{p-1}(1,2)+pH_{p-1}(3)\\
&\quad +2p^2H_{p-1}(1,1,2)+p^2H_{p-1}(2,2)
+p^2H_{p-1}(1,3)\\
&\quad+3p^3H_{p-1}(1,1,1,2)+p^3H_{p-1}(2,1,2)+p^3H_{p-1}(1,2,2)\\
&\quad +p^3 H_{p-1}(1,1,3)\\
&\equiv -\frac{H_{p-1}(1)}{p}-\frac{3pH_{p-1}(3)}{2}\pmod{p^4}
\end{align*}
where at the last step we applied the results mentioned in the preliminaries.

By comparing the coefficient of $a^2$ in the expansion  of both sides of \eqref{ID1} at $a=0$ we have
$$\sum_{k=1}^n
\binom{2k}{k}\frac{5k^2}{4n^4+k^4}
\prod_{j=1}^{k-1}
\frac{n^4-j^4}{4n^4+j^4}
\left(
\frac{1}{5k^2}+\sum_{j=1}^{k-1}\frac{1}{n^2+j^2}
-2\sum_{j=1}^{k}\frac{2n^2+j^2}{4n^4+j^4}
\right)
=-\frac{2}{n^4}.$$
Let $n=p>3$ be a prime then move to the right-hand side the $p$-th term of the sum on the left. Thus, the left-hand side modulo $p^2$ is congruent to
$$\sum_{k=1}^{p-1}
(-1)^{k-1}\binom{2k}{k}\frac{5}{k^2}
\left(
\frac{1}{5k^2}+H_{k-1}(2)
-2H_{k}(2)
\right)=\sum_{k=1}^{p-1}(-1)^k\binom{2k}{k}\left(\frac{4}{k^4}+\frac{5H_k(2)}{k^2}\right).$$
The right-hand side multiplied by $p^4$ is 
\begin{equation}\label{rhs}
-2-\binom{2p}{p}
\prod_{j=1}^{p-1}
\frac{p^4-j^4}{4p^4+j^4}
\left(
\frac{1}{5}+p^2\sum_{j=1}^{p-1}\frac{1}{p^2+j^2}
-2p^2\sum_{j=1}^{p}\frac{2p^2+j^2}{4p^4+j^4}
\right)
\end{equation}
and it remains to verify that it is congruent to $-p^4H_{p-1}(4)$ modulo $p^6$.
We note that
\begin{align*}
&\prod_{j=1}^{p-1}
\frac{p^4-j^4}{4p^4+j^4}\equiv 1-5p^4H_{p-1}(4)
\pmod{p^6},\\
&p^2\sum_{j=1}^{p-1}\frac{1}{p^2+j^2}\equiv p^2H_{p-1}(2)-p^4H_{p-1}(4)
\pmod{p^6},\\
&2p^2\sum_{j=1}^{p}\frac{2p^2+j^2}{4p^4+j^4}
\equiv\frac{6}{5}+2p^2H_{p-1}(2)+4p^4H_{p-1}(4)
\pmod{p^6},
\end{align*}
Hence, by \eqref{ppp}, \eqref{rhs}  simplifies to
$$
-2+2\left( 1-p^2H_{p-1}(2)-\frac{p^4H_{p-1}(4)}{2}\right)
\left( 1+p^2H_{p-1}(2)\right)
\equiv -p^4H_{p-1}(4)\pmod{p^6}$$
and the proof is finished.
\end{proof}

\end{document}